\documentclass[12pt]{amsart}

\usepackage{amsthm}
\theoremstyle{plain}

\makeatletter \@addtoreset{equation}{section}

\makeatother \makeatletter
\newtheorem{theorem}{Theorem}[section]
\newtheorem*{theorem*}{Theorem}
\newtheorem{corollary}[theorem]{Corollary}
\newtheorem{definition}[theorem]{Definition}

\newtheorem{proposition}[theorem]{Proposition}
\newtheorem*{Proposition}{Proposition}
\newtheorem{remark}[theorem]{Remark}

\newtheorem{example}[theorem]{Example}

\def\pr1{\prod\hskip -2.07ex * \hskip 0.9 ex}

\usepackage{latexsym} \usepackage{amsmath} \usepackage{amssymb}
\usepackage{amsthm}

\usepackage{cancel}\usepackage{color}
\begin{document}
\noindent
\title{Zero sets and  Nullstellensatz type theorems for slice regular quaternionic polynomials}
\author{Anna Gori $^1$} 
 \thanks{$^1$ Dipartimento di Matematica - Universit\`a di Milano,                
 Via Saldini 50, 20133  Milano, Italy ({anna.gori@unimi.it})}
 \author{Giulia Sarfatti $^2$}
 \thanks{ $^2$ DIISM - Universit\`a Politecnica delle Marche,               
 	Via Brecce Bianche 12,  60131, Ancona, Italy ({g.sarfatti@univpm.it})}
 \author{Fabio Vlacci $^3$}
 \thanks{$^3$ DiSPeS  - Universit\`a di Trieste, Piazzale Europa 1, 34100, Trieste, Italy ({fvlacci@units.it})}

 \begin{abstract}
   We study the vanishing sets of slice regular polynomials in several
   quaternionic variables.  
  We obtain a geometric description of the vanishing sets in two variables, which leads to a
   new version of the {\em Strong} Hilbert Nullstellensatz in the quaternionic setting.

 \end{abstract}
\maketitle
\noindent
\section*{Acknowledgments}
 The authors are partially supported by:
GNSAGA-INdAM via the project ``Hypercomplex function theory and applications''; by MUR, via the project Finanziamento Premiale FOE 2014 ``Splines for accUrate NumeRics: adaptIve models for Simulation Environments''; the first author is also partially supported by MUR project PRIN 2022 ``Real and Complex Manifolds: Geometry and Holomorphic Dynamics'', the second and the third authors are also partially supported by MUR project PRIN 2022 ``Interactions between Geometric Structures
and Function Theories''.
\section{Introduction}

The so called {\em Weak} Hilbert  Nullstellensatz can be regarded as a generalization of
the Fundamental Theorem of Algebra to the case of polynomials in
several complex variables, as pointed out in \cite{manetti}.   It has also a {\em Strong} formulation which
provides a correspondence between radical ideals and vanishing sets of polynomials. This theorem represents a central tool in many active research fields in
Mathematics, and especially in Algebraic Geometry. 
In the complex setting these two versions of the Hilbert Nullstellesantz, are actually equivalent thanks to the fact that point-evaluation of a
polynomial is a ring homomorphism of complex polynomials.

The Hilbert Nullstellensatz is generally
stated in the framework of algebraically closed fields, but in recent
times some new interest has been addressed to a
formalization of the Nullstellensatz in a noncommutative setting
(see, e.g., \cite{israeliani-1, israeliani, BR}).  In particular, Alon
and Paran in \cite{israeliani-1} proved both a Weak and a Strong
version of the Nullstellensatz in the ring $\mathbb{H}[x_1,\ldots,x_n]$ of quaternionic polynomials with {\em central}
variables, i.e. such that $x_\ell x_m=x_m x_\ell$  for all $\ell,m=1\dots n$. 
We remark here that, in this framework, the equivalence of the two versions of the Nullstellensatz cannot be shown using the point-evaluation of polynomials since it is not a ring homomorphism in general.  

In the last decade the theory of slice regular 
quaternionic functions has proved to be central for the development of
the study of quaternionic maps which resemble the main properties of
holomorphic functions in the complex setting (see,
e.g., \cite{ libroGSS,GS}).  Furthermore, even though an
analog of the Fundamental Theorem of Algebra does not hold in general
for polynomials over the quaternions, in \cite{GSV} a positive result
in this sense has been obtained for slice regular polynomials over
quaternions and octonions.  Thus it is quite natural to look for a
version of the Nullstellensatz in the framework of slice regular
polynomials  in $n$ quaternionic variables $\mathbb{H}[q_1, \ldots, q_n]$.\\ It is possible to
establish an isomorphism $\varphi$ between the ring $\mathbb{H}[q_1, \ldots, q_n]$ equipped with the so called {\em $*$-product}
(see Section \ref{sec1}) and the ring of 
polynomials in several {central} variables $\mathbb{H}[x_1, \ldots, x_n]$ equipped with the standard
pointwise product, considered in \cite{israeliani-1}. Thanks to the isomorphism $\varphi$, it is immediate to rephrase in our setting the weak version of the quaternionic Hilbert Nullstellensatz proved in \cite{israeliani-1}.
However, the established isomorphism cannot be directly used to study the vanishing sets of polynomials. Indeed,
an element in $\mathbb{H}[x_1, \ldots, x_n]$ can be evaluated as a
function only on $n$-tuples $(a_1,\ldots,a_n)\in\mathbb{H}^n$ with
commuting components $a_la_m=a_ma_l$ for any $l,m=1,\ldots,n$ (a
nowhere dense subset of $\mathbb{H}^n$)
%
%
whereas the corresponding regular polynomial function is well defined in $\mathbb{H}^n$.

In order to study the zero set of slice regular polynomials, we begin by focusing our attention to their factorization properties in terms of the {\em $*$-product}; in particular we prove

	\begin{Proposition}
	Let $P\in \mathbb{H}[q_1, \ldots, q_n]$
	be a slice regular polynomial in $n$ variables and let $1\le m \le n$. Then $P$ vanishes on $\mathbb{H}^{m-1}\times\{a\}\times (C_a)^{n-m}$
if and only if there exists $P_m\in\mathbb{H}[q_1, \ldots, q_n]$ such that 
\[P(q_1, \ldots, q_n)=(q_m-a)*P_m(q_1, \ldots, q_n),
\]	
where $C_a:=\{q\in\mathbb{H}\ |\ aq=qa\}$.
        \end{Proposition}
 While polynomials vanishing at points with commuting components are studied in \cite{israeliani-1}, we obtain the following  result for slice regular polynomials vanishing at generic points.
 
\begin{Proposition}
	A slice regular polynomial $P\in
	\mathbb{H}[q_1,\ldots,q_n]$ vanishes at $(a_1,\ldots,a_n)\in \mathbb{H}^n$ if and only if there exist
	$P_k\in \mathbb H[q_1,\ldots,q_k]$ for any $k=1,\ldots,n$ such that
	\begin{equation}
P(q_1,\ldots,q_n)=\sum_{k=1}^n(q_k-a_k)*P_k(q_1, \ldots,q_k).
	\end{equation}	
\end{Proposition}

 Let $I$ be a right ideal  in $\mathbb{H}[q_1,\ldots,q_n]$; we
define $\mathcal{V}(I)$ to be the set of common zeros of polynomials
in $I$.  Let $Z$ be 
a subset of $\mathbb{H}^n$, 
we denote by
$\mathcal{I}(Z)$ the right ideal given by the intersection,  for ${(a_1,\ldots,a_n)\in Z},$ of the right ideals $\mathcal{I}_{(a_1,\ldots,a_n)}$ generated, via the
$*$-product, by $q_1-a_1, q_2-a_2, \ldots,q_n-a_n$ in $\mathbb{H}[q_1,\ldots, q_n]$. \\We point out that, in the
quaternionic setting, $\mathcal{I}({Z})$ does not always coincide
with the set of polynomials vanishing on $Z$.  Indeed the set of polynomials whose zero locus is  $Z$ is not an ideal, in general.  So it
is natural to consider also the ideal $\mathcal{J}(Z)$ generated by
polynomials vanishing on $Z$;  we then   investigate the relations of $\mathcal{J}(Z)$ with
$\mathcal{I}(Z)$.
The two sets $\mathcal{I}(Z)$ and $\mathcal{J}(Z)$ coincide, for instance,  when $Z$ consists only of points
with commuting components (this is the case considered  in \cite{israeliani-1}). 
In the general case, we can show the inclusion $\mathcal{J}(Z)\subseteq \mathcal{I}(Z)$.

\noindent After introducing the radical of an ideal $I$ as the
intersection of all {\em completely prime} ideals containing $I$, 
we can rephrase
the Strong version of the Nullstellensatz proved in \cite{israeliani-1} in 
the setting of slice regular polynomials.

In the two variable case, 
we  show some relevant geometric properties of the 
 vanishing set of an ideal in  $\mathbb{H}[q_1,q_2]$.
  Among other results, we prove 
  \begin{Proposition}
 A slice regular polynomial $P\in \mathbb{H}[q_1,q_2]$
vanishes on $\{a\}\times \mathbb{H}$ if and only if   there exists a polynomial $ Q\in\mathbb{H}[q_1]$
such that $P(q_1,q_2)=(q_1-a)*Q(q_1)$. 
\end{Proposition}
If  $a\in\mathbb{H}\setminus\mathbb{R}$, we define  $\mathbb{S}_a:=\{cac^{-1} \ :\ |c|=1\} $. 
 Then
\begin{Proposition}
A slice regular  polynomial $P\in \mathbb{H}[q_1,q_2]$ vanishes on $\mathbb{S}_a\times \{b\}$ if and only
	if there exist $P_1\in \mathbb{H}[q_1]$ and $P_2\in \mathbb{H}[q_1,q_2]$ such that
	\begin{align*}
	P(q_1,q_2)&=(q_1^2-2{\rm Re}(a)q_1+|a|^2)*P_1(q_1)+(q_2-b)*P_2(q_1,q_2) 
\end{align*}	
\end{Proposition}
We then say 
that a subset $D\subseteq \mathbb{H}^2$ is  {\em
          $q_1$-symmetric} if, for any $(a,b)\in  D$, the set
        $\mathbb S_a \times \{b\}$ is contained in $D$. With this notation, we have 

\begin{Proposition} 
	Given an ideal $I\subseteq \mathbb{H}[q_1,q_2]$, either $\mathcal V(I)$  consists only of points with commuting components or $\mathcal V(I)$ is $q_1$-symmetric. \end{Proposition}
%
This is the key ingredient to show that $\mathcal{J}(\mathcal{ V }(I))$ coincides
with $\mathcal{I}(\mathcal{ V }(I))$ in $\mathbb{H}[q_1,q_2]$
and thus to have a more 
      geometric interpretation of the Strong Nullstellensatz in this
      framework. 
       \begin{theorem*}[Strong Nullstellensatz in $\mathbb{H}^2$]
Let $I$ be a right ideal in $\mathbb{H}[q_1,q_2]$. Then
\[\mathcal{J}(\mathcal{V}(I))=\sqrt{I}.\] 	
Moreover,  $\sqrt I$ coincides with the ideal of polynomials vanishing on $\mathcal V(I)$.
\end{theorem*}
These promising results  are in the direction of developing
  a theory of quaternionic algebraic varieties. \\
The present paper is organized as follows:
 in Section \ref{sec1} we shortly recall the main definitions and results from the theory of slice 
regular polynomial functions which will be
used in the sequel. The factorization
of slice regular polynomials and application to the study of ideals in $\mathbb{H}[q_1\ldots,,q_n]$ is treated in Section \ref{333}, where
moreover one can find  several properties of the vanishing sets of slice regular polynomials in two quaternionic variables.
 The Nullstellensatz type theorems 
 for slice regular polynomials
 are investigated
 in Section \ref{WeNu}; {
in particular, a new version of the Strong Nullstellensatz  Theorem  is proved for ideals in $\mathbb{H}[q_1,q_2]$.
Finally we provide examples  in several quaternionic variables that enforce the evidence that this new version of the Strong Nullstellensatz
should hold in $\mathbb{H}[q_1, \ldots, q_n]$.

\section{Introduction to quaternionic slice regular polynomials}\label{sec1}
Let $\mathbb{H}{=\mathbb{R}+i\mathbb{R}+j\mathbb{R}+k\mathbb{R}}$ denote the skew field of
quaternions and let $\mathbb{S}=\{q \in \mathbb{H} \ : \ q^2=-1\}$ be
the two dimensional sphere of quaternionic imaginary units.  Then
\[ \mathbb{H}=\bigcup_{J\in \mathbb{S}}(\mathbb{R}+\mathbb{R} J),  
\]
where the ``slice'' $\mathbb{C}_J:=\mathbb{R}+\mathbb{R} J$ can be
identified with the complex plane $\mathbb{C}$ for any
$J\in\mathbb{S}$.  { In this way, any $q\in \mathbb{H}$ can be
  expressed as $q=x+yJ$ with $x,y \in \mathbb{R}$ and $J \in
  \mathbb{S}$. The {\it real part }of $q$ is ${\rm Re}(q)=x$ and its
         {\it imaginary part} is ${\rm Im}(q)=yJ$; the {\it conjugate}
         of $q$ is $\bar q:={\rm Re}(q)-{\rm Im}(q)$.  For any non-real
         quaternion $a\in \mathbb{H}\setminus \mathbb{R}$ we will
         denote by $J_a:=\frac{{\rm Im}(a)}{|{\rm Im}(a)|}\in
         \mathbb{S}$ and by $\mathbb{S}_a:=\{{\rm Re}(a)+J |{\rm
           Im}(a)| \ : \ J\in \mathbb{S}\}$. If $a\in \mathbb{R}$,
         then $J_a$ is any imaginary unit.}\\ 
 In the present paper,  the central object of our study are slice regular quaternionic polynomial functions of the form
 \[q\mapsto \sum_{n=0}^{N}q^na_n,  
 \  a_n \in \mathbb{H};\]
 $N$ will represent the {\em degree} of the polynomial function.
While the sum of two slice regular polynomials  is clearly slice regular, an appropriate notion of
multiplication of slice regular polynomials is given by the so called {\em $*$-product}.
\begin{definition}
	If $P(q)=\sum_{n=0}^{N} q^n a_n$ and
        $Q(q)=\sum_{n=0}^{M} q^n b_n$ are two slice regular polynomials, then the $*$-product of $P$ and $Q$ is  defined by
	$$
	P\ast Q(q):=\sum_{n=0}^{N+M}q^n\sum_{k=0}^na_k b_{n-k}.
	$$
	\end{definition}
\noindent Notice that the $*$-product is associative but not commutative in general.

\noindent It is not difficult to see that if $P(q)=\sum_{n=0}^{N}q^na_n$ has real coefficients $\{a_n\}$ (i.e. $P$ is {\em
  slice preserving}), then, for any slice regulr polynomial $Q$,
$
P*Q(q)=
P(q)\cdot Q(q)=
Q*P(q).
$

\noindent
In general, the relation of the $*$-product with the usual pointwise product is the following (see \cite[Theorem 3.4]{libroGSS} ):
	$$
	P* Q(q)=\left\{\begin{array}{c cl}
	0 & &\text{if $P(q)=0 $}\\
	P(q)\cdot Q(P(q)^{-1}\cdot q \cdot P(q)) & &\text{if $P(q)\neq 0 $\,,}
	\end{array}\right.
	$$
	for any $P,Q$ slice regular polynomials.
	\noindent
        Notice that $P(q)^{-1}\cdot q\cdot P(q)$ belongs to the 
        sphere $\mathbb{S}_q$. Hence each zero of $P*Q$ in
        $\mathbb{S}_q$ is given either by a zero of $P$ or by a point
        which is a conjugate of a zero of $Q$ in the same sphere.
	
A peculiar aspect of  slice regular polynomials (more in general slice regular functions) is the structure of their zero
sets. In fact, besides isolated zeros, they can also vanish on two
dimensional spheres. As an example, the polynomial $q^2+1$ vanishes on
the entire sphere of imaginary units $\mathbb{S}$.  It can be proven
that also spherical zeros cannot accumulate.

\begin{theorem}\cite[Theorem 3.13]{libroGSS}
	Let $P$ be a slice regular polynomial. If $P$ does not vanish identically, then its zero
        set consists of isolated points or isolated $2$-spheres of the
        form $x +y \mathbb{S}$ with $x,y \in \mathbb{R}$, $y \neq 0$.
\end{theorem}
\noindent Slice regular polynomials can be viewed as an important subclass of slice regular functions. We refer to \cite{libroGSS} for a detailed survey on the theory of slice regular functions in one quaternionic variable. }
\subsection{Several variable case}

We will deal with polynomial functions $P: \mathbb{H}^n\to \mathbb{H}$ of the form 
	\[P(q_1,\ldots, q_n)=\sum_{ \substack{\ell_1=0,\ldots, L_1  \\
	\cdots\\ \ell_n=0,\ldots, L_n} }{q_1}^{\ell_1}\cdots {q_n}^{\ell_n}a_{\ell_1,\ldots,\ell_n} \]
	with $a_{\ell_1,\ldots,\ell_n}\in\mathbb{H}$, 
where $\deg_{{q_{\ell_j}}}P:=L_j$.\\	
\noindent
These polynomial functions are  examples of slice regular functions on $\mathbb{H}^n$. %
When considering functions of several quaternionic variables, the
definition of slice regularity 
relies on
the notion of {\em stem functions}.
The formulation of the theory of slice regular funcions in several quaternionic variables 
can be found in \cite{severalvariables}. 
For the sake of simplicity, we will give the main definitions and statements in the case of slice regular polynomial functions
in two quaternionic variables.
All the notions and properties discussed in this section can be easily adapted to the several variables case.

%

Slice regular polynomial functions of several variables can be endowed with an
appropriate notion of product, the so called {\em slice product}. It
will be denoted by the same symbol $*$ used in the one-variable
case. 
Let us recall here how it works for
slice regular polynomials in two variables.
\begin{definition}
	If $P(q_1,q_2)=\sum_{ \substack{n=0,\ldots,N_1  \\ m=0,\ldots,N_2} } q_1^nq_2^m a_{n,m}$ and
        $Q(q)=\sum_{ \substack{n=0,\ldots,L_1  \\ m=0,\ldots,L_2} }q_1^nq_2^m b_{n,m}$ are two
        polynomials,
        then the {\em $*$-product} of $P$ and $Q$ is the polynomial
        defined by
$$
P*Q (q_1,q_2):=\sum_{ \substack{n=0,\ldots,N_1+L_1  \\ m=0,\ldots,N_2+L_2}}q_1^nq_2^m\sum_{ \substack{r=0,\ldots,n  \\ s=0,\ldots,m} }a_{r,s} b_{n-r,m-s}
$$
\end{definition}
\noindent For example, if $a,b \in \mathbb{H}$, then
\begin{itemize}
	\item $q_1*q_2=q_2*q_1=q_1q_2$;
	\item $a*(q_1q_2)=(q_1q_2)*a=q_1q_2a$;
	\item $(q_1^nq_2^ma)*(q_1^rq_2^sb)=q_1^{n+r}q_2^{m+s}ab$.
\end{itemize} 	
\noindent moreover we point out that, if $P$ or $Q$ have real coefficients, then $P*Q=Q*P$.

{
\begin{remark}\label{valutazione}
Observe that 
the evaluation of slice regular power series
is not a multiplicative homomorphism.
Therefore, the zeros  of the $*$-product of two slice regular polynomials is not in general the union
    of the zeros of each of the factors.
For instance, $q_1-i$ vanishes on $\{i\}\times \mathbb{H}$,
    while $(q_1-i)*(q_2-j)=q_1q_2-q_1j-q_2i+k$, when $q_1=i$, vanishes only for  $q_2\in\mathbb{C}_i$.

\end{remark}
}


\medskip 

\noindent In the sequel we will denote by $\mathbb{H}[q_1,\ldots,q_n]$ the set of slice regular polynomials in $n$ quaternionic variables. Since the $*$-product is associative but not commutative, $(\mathbb{H}[q_1,\ldots,q_n], +, *)$ is
a noncommutative ring (without zero divisors). 

\begin{definition}
	A subset $ I$ of $\mathbb{H}[q_1,\ldots,q_n]$, closed under addition, is called
	\begin{itemize}
		\item a {\em left ideal} if for any $P\in\mathbb{H}[q_1,\ldots,q_n]$, $P* I=\{P*Q \ : \ Q \in  I\} \subseteq  I$;
		\item a {\em right ideal} if for any $P\in\mathbb{H}[q_1,\ldots,q_n]$, $I*P=\{Q*P \ : \ Q \in I\} \subseteq  I$;
		\item a {\em two-sided ideal} if $ I$ is both a left and a right ideal. 
	\end{itemize}
\end{definition}

From the algebraic point of
view,
the 
ring $(\mathbb{H}[q_1,\ldots,q_n], +, *)$ is
isomorphic to the ring $(\mathbb{H}[x_1,\ldots, x_n], +, \cdot)$ of
quaternionic polynomials in several central variables with left coefficients considered in
\cite{israeliani-1}, via the map defined on monomials as
\begin{equation}\label{iso}
	\begin{aligned}
		\varphi: (\mathbb{H}[q_1,\ldots,q_n], +, *) &\longrightarrow (\mathbb{H}[x_1,\ldots, x_n], +, \cdot)\\
	\varphi: q_1^{\ell_1}\cdots q_n^{\ell_n}a &\longmapsto \overline{a} x_1^{\ell_1}\cdots x_n^{\ell_n},
	\end{aligned}
\end{equation}
and then extended by additivity to polynomials.
The Identity Principle for Polynomials guarantees that $\varphi$ is a bijection. Moreover it satisfies the equality
\begin{equation*}
\begin{aligned}
	\varphi& (q_1^{\ell_1}\cdots q_n^{\ell_n}a*q_1^{k_1}\cdots q_n^{k_n}b)=\varphi(q_1^{\ell_1+k_1}\cdots q_n^{\ell_n+k_n}ab)= 
	\overline{ab} x_1^{\ell_1+k_1}\cdots x_n^{\ell_n+k_n}=
	\\=&\overline{b} \overline{a} x_1^{\ell_1+k_1}\cdots x_n^{\ell_n+k_n}
	=\overline{b} x_1^{\ell_1}\cdots x_n^{\ell_n}\cdot \overline{a} x_1^{k_1}\cdots x_n^{k_n}=\varphi(q_1^{k_1}\cdots q_n^{k_n}b) \cdot \varphi(q_1^{\ell_1}\cdots q_n^{\ell_n}a).
\end{aligned}	
 \end{equation*}
 { The isomorphism $\varphi$ inverts the order of the factors, this is in accordance with the fact that ideals considered in \cite{israeliani-1} are {\em left} ideals, while in our context it is natural to consider {\em right} ideals.
 }
 
{\begin{remark}\label{isomorph} 
Observe that, given $P\in \mathbb{H}[q_1,\ldots,q_n], +, *)$, as a slice regular polynomial funtion, it is defined in  $\mathbb{H}^n$, whereas $\varphi(P)\in (\mathbb{H}[x_1,\ldots, x_n], +, \cdot)$ 
can be regarded as a function only defined in
    $\bigcup_{J\in\mathbb{S}} (\mathbb{C}_J)^n
    \Subset
    \mathbb{H}^n$.

  
\end{remark}
}

{

 As a consequence of the previous remark, the isomorphism $\varphi$  cannot be used to have information on the evaluations of the isomorphic polynomials in the two polynomial rings.
 However, 
in $\bigcup_{J\in\mathbb{S}} (\mathbb{C}_J)^n
\Subset
    \mathbb{H}^n$
it is possible to establish a correspondence between values of slice regular polynomials in $\mathbb{H}[q_1,\ldots,q_n]$ and those  of their images in $\mathbb{H}[x_1,\ldots,x_n]$ through $\varphi$. Indeed,
let $(a_1,\ldots,a_n) \in \mathbb{H}^n$ be such that $a_\ell a_m=a_ma_\ell$ for any $\ell,m$, and let $P\in\mathbb{H}[q_1,\ldots,q_n]$. Then, it is immediate to see that 

\begin{equation}\label{valori}
	\varphi(P(a_1,\ldots,a_n))=(\varphi(P))(\overline{a_1},\ldots,\overline{a_n}).
	\end{equation} }

As in the one-variable case, it is possible to introduce two
        operators on slice regular polynomials (see \cite[Definition 1.46]{libroGSS}). For the sake of
        simplicity, we define them only in the case of two
        quaternionic variables.
\begin{definition}\label{R-coniugata2}
	Let $P(q_1,q_2)=\sum\limits_{\substack{n=0,\ldots,N  \\ m=0,\ldots,M}}q_1^nq_2^ma_{n,m}$ be
        a slice regular polynomial. 
Then the \emph{regular conjugate} of $P$ is the
        slice regular polynomial 
	\[P^c(q_1,q_2)=\sum_{\substack{n=0,\ldots,N  \\ m=0,\ldots,M}}q_1^nq_2^m\overline{a_{n,m}},\]
	and the
	\emph{symmetrization} of $P$ is the slice regular polynomial 
	\[P^s= P*P^c = P^c*P.\]
\end{definition}
 A very useful result, that will be used in the sequel,
  is the following version of {\em Identity Principle} for slice
  regular polynomial functions. See \cite[Corollary 2.13]{severalvariables}
  for the details.

\begin{theorem}[Identity Principle]\label{Id}
	Let $P, Q$ be two slice regular polynomials in $n$
        quaternionic variables. If there exist
        $J_1,\ldots,J_n \in \mathbb S$ such that $P\equiv Q$ on
        $
        \mathbb{C}_{J_1}\times\cdots\times\mathbb{C}_{J_n}$, then
        $P\equiv Q$ on $\mathbb{H}^n$.
\end{theorem}
\section{Some remarks on factorization of slice regular polynomials}\label{333}

In what follows we will only consider slice regular polynomials in quaternionic variables. For brevity, we will refer to them simply as {\em regular polynomials}.
 In the one-variable setting it is possible to perform both left and right Euclidean $*$-division between regular polynomials, see \cite[Proposition 3.42]{libroGSS}. For regular polynomials in several quaternionic variables, as in the complex case, we need to restrict our setting to the case of division of regular polynomials by a {\em monic} regular  polynomial.

\begin{definition}
A  regular polynomial $P(q_1,\ldots, q_2)$ is {\em monic} of degree $d$ in the variable
$q_j$ if it can be written as
\[	P(q_1,\ldots, q_n)=q_j^d+\sum_{k=1}^{d-1}q_j^{k}*P_k(q_1,\ldots,q_{j-1},q_{j+1}, \ldots,q_n)\] 
with $P_1,\ldots,P_{d-1} \in \mathbb{H}[q_1,\ldots,q_{j-1},q_{j+1}, \ldots, q_n]$.
\end{definition}
\begin{proposition}\label{euclidean}
	Let $M\in \mathbb{H}[q_1,\ldots,q_n]$ be a monic regular polynomial
        of degree $d$ in $q_j$, with $1\le j \le n$. Then, for any $P\in
        \mathbb{H}[q_1,\ldots,q_n]$, there exist, and are unique, regular polynomials
        $Q\in \mathbb{H}[q_1,\ldots,q_n]$ and $R_0, \ldots,
        R_{d-1}\in \mathbb{H}[q_1, \ldots,q_{j-1},q_{j+1},\ldots,q_n]$ such that
	\[P=M*Q+\sum_{k=0}^{d-1}q_j^k*R_{k}.\]
\end{proposition}  
\begin{proof}
Let $P_0,\ldots,P_{s}\in \mathbb{H}[q_1,\ldots,q_{j-1},q_{j+1}, \ldots, q_n]$ be such that
\[P=\sum_{k=0}^{s}q_j^k*P_{k},\]
and proceed by induction on the degree $s$ of $P$ in $q_j$.

\noindent
If $s<d$, then we set $Q\equiv 0$ and $R_{k}=P_k$ and we immediately
prove the statement.\\ 
Otherwise, consider $\hat P = P- M*q_j^{s-d}*P_s$.
Since $M$ is monic of degree $d$ in $q_j$, we get that $\deg_{q_j}\hat P<s$. We can therefore use the induction hypothesis to write 
$\hat P=M*\hat Q+\sum\limits_{k=0}^{d-1}q_j^k\hat
R_{k}$, so that
\[ P=\hat P + M*q_j^{s-d}*P_s=M*(\hat Q+q_j^{s-d}*P_s)+\sum_{k=0}^{d-1}q_j^k\hat R_{k}.\]  
Thus, setting $Q=\hat Q+q_j^{s-d}*P_s$ and $R_{k}=\hat R_k$, we conclude.
\end{proof}


           Here we begin the study of zeros of regular polynomials. 
           Our interest is twofold: on the one hand we aim to better understand the structure of the vanishing sets of regular polynomials, on the other hand we want to investigate the relation between such vanishing sets and ideals in $\mathbb{H}[q_1,\ldots, q_n]$. 

\begin{definition}
If $a\in\mathbb{H}$, let $C_a$ be the set of $q\in \mathbb{H}$ such that $aq=qa$, namely
\[ C_a=\left\{\begin{array}{rcl} \mathbb{C}_{J_a} &\mathrm{if}& a\in\mathbb{H}\setminus\mathbb{R}\\\\
\mathbb{H} &\mathrm{if}& a\in\mathbb{R}\end{array}
\right.\]

\end{definition}
%

  Let us first consider a simple example that enlights 
   how the lack of commutativity of the product affects the
  structure of the zero locus of a regular polynomial in two
  quaternionic variables and shows that these two variables are not interchangeable.
	\begin{example}
		Let $a,b\in \mathbb{H}$ be such that $ab\neq ba$ and consider the regular polynomials
		\[P(q_1,q_2)=(q_1-a)*(q_2-b)=q_1q_2-q_1b-q_2a+ab\]
		and 
		\[Q(q_1,q_2)=(q_2-b)*(q_1-a)=q_1q_2-q_1b-q_2a+ba.\]
For $q_1=a$ we have 
\[P(a,q_2)=aq_2-q_2a \quad \text{and} \quad Q(a,q_2)=aq_2-ab-q_2a+ba.\] 
Hence $P$ vanishes on $\{a\}\times C_a$, while $Q$ does not. \\
On the other hand, for $q_2=b$, 
\[P(q_1,b)=-ba+ab \quad \text{and} \quad Q(q_1,b)\equiv 0.\]
Hence $Q$ vanishes on $\mathbb{H}\times\{b\}$, while $P$ is never zero when $q_2=b$.  
\end{example}
%


{	\begin{remark}\label{noncorrisp}

		In addition to Remark \ref{isomorph}, observe that there is not a direct
		correspondence between the zero locus of regular polynomials in
		$(\mathbb{H}[q_1,q_2], +, *)$ with the zero locus of polynomials in
		$(\mathbb{H}[x_1,x_2], +, \cdot)$.
		In fact, the zero locus of the regular
		polynomial $P(q_1,q_2)=(q_2-i)*(q_1-j)$ in $\mathbb{H}^2$ contains
		 $\mathbb{H} \times \{i\}$, while its isomorphic image via $\varphi$ in
		$(\mathbb{H}[x_1,x_2], +, \cdot)$, namely $(x_1+j)\cdot (x_2+i)$,
		for $x_2=-i$ vanishes only when $x_1 \in\mathbb{C}_{i}$ (since it cannot be
		evaluated at points with noncommutative coordinates).

	\end{remark}
}

     The next proposition gives a first geometrical  description of the zero set of regular polynomials in several quaternionic variables.
        
	\begin{proposition}\label{mlineare}
	Let $P\in \mathbb{H}[q_1, \ldots, q_n]$
	be a regular polynomial in $n$ variables and let $1\le m \le n$. Then $P$ vanishes on $\mathbb{H}^{m-1}\times\{a\}\times (C_a)^{n-m}$
	if and only if there exists $P_m\in\mathbb{H}[q_1, \ldots, q_n]$ such that 
	\[P(q_1, \ldots, q_n)=(q_m-a)*P_m(q_1, \ldots, q_n).\]	
\end{proposition}
\begin{proof}
	Let $P_m(q_1,\ldots, q_n):=
	\sum\limits_{\substack{k_j=0,\ldots,N_j  \\ j=1,\ldots,n}}q_1^{k_1}\cdots q_n^{k_n} b_{k_1\ldots k_n}$ be a regular polynomial in  $\mathbb{H}[q_1,\ldots, q_n]$
	and consider $P=(q_m-a)*P_m$. Then 
	
	\begin{eqnarray*}
		P(q_1,\ldots, q_n)&=&q_m*P_m(q_1,\ldots, q_n)-a*P_m(q_1,\ldots, q_n)\\
		&=&\sum \limits_{\substack{k_j=0,\ldots, N_j  \\ j=1,\ldots,n}} q_1^{k_1}\cdots q_m^{k_m+1} \cdots q_n^{k_n} b_{k_1 \ldots k_n}-\sum\limits_{\substack{k_j=0,\ldots,N_j  \\ j=1,\ldots,n}}q_1^{k_1}\cdots q_m^{k_m} \cdots q_n^{k_n}\cdot a \cdot b_{k_1\ldots k_n},	
	\end{eqnarray*}
	and hence, for any $(q_1,\ldots, q_{m-1}, a, u_{m+1},\ldots, u_n)\in \mathbb{H}^{m-1}\times\{a\}\times (C_a)^{n-m}$,
		\begin{eqnarray*}
		&&P(q_1,\ldots, q_{m-1}, a, u_{m+1},\ldots, u_n)\\ 
		&=&\sum\limits_{\substack{k_j=0,\ldots,N_j  \\ j=1,\ldots,n}}q_1^{k_1}\cdots a^{k_m+1} \cdots u_n^{k_n} \cdot b_{k_1\cdots k_n}-\sum\limits_{\substack{k_j=0,\ldots,N_j  \\ j=1,\ldots,n}}q_1^{k_1}\cdots a^{k_m}\cdots u_n^{k_n}\cdot a\cdot b_{k_1\cdots k_n}\\
		&=&0.	
	\end{eqnarray*}
	
	On the other hand, suppose that a regular polynomial $P$ vanishes on $\mathbb{H}^{m-1}\times\{a\}\times (C_a)^{n-m}$. Then, performing the $*$-division of $P$ by $q_m-a$ as in Proposition \ref{euclidean}, we obtain
	\[P(q_1,\ldots, q_n)=(q_m-a)*P_m(q_1,\ldots, q_n)
	+R(q_1,q_2,\ldots, q_{m-1},q_{m+1},\ldots, q_n).\]
	Now, for any $(q_1,\ldots, q_{m-1}, a, u_{m+1},\ldots, u_n)\in \mathbb{H}^{m-1}\times\{a\}\times (C_a)^{n-m}$ we have that
		$P(q_1,\ldots, q_{m-1}, a, u_{m+1},\ldots, u_n)=0$
	and, as in the previous considerations, also
	\[[(q_m-a)*P_m(q_1,\ldots, q_n)]|_{(q_1,\ldots, q_{m-1}, a, u_{m+1},\ldots, u_n)}=0.\]
Hence $R$ vanishes identically on $\mathbb{H}^{m-1}\times (C_a)^{n-m}$.
		Thanks to the Identity Principle \ref{Id}, we get that $R$ is identically zero and we conclude the proof.
	\end{proof}

We now want to establish a relation between zeros of regular polynomials and ideals in $\mathbb{H}[q_1,\ldots,q_n]$. 
	If $(a_1,\ldots,a_n)\in {\mathbb{H}}^n$, then denote by
\[\mathcal I_{(a_1,\ldots,a_n)}=\{(q_1-a_1)*P_1+\cdots+(q_n-a_n)*P_n \ | \ P_1,\ldots, P_n \in\mathbb{H}[q_1,\ldots,q_n]\}.\]
Thanks to the properties of the $*$-product, it is not difficult to
prove that $\mathcal I_{(a_1,\ldots,a_n)}$ is a right ideal, generated (via the
$*$-product) by $q_1-a_1, q_2-a_2, \ldots,q_n-a_n$ in $\mathbb{H}[q_1,\ldots, q_n]$.
{\begin{remark}\label{idealicorrispondenti}
 The isomorphism $\varphi$ introduced in \eqref{iso} maps the ideal $\mathcal I_{(a_1,\ldots,a_n)}\subseteq \mathbb H[q_1,\ldots,q_n]$ to the ideal $I_{\overline{a}}$ in $\mathbb H[x_1,\ldots,x_n]$, where $\overline{a}=(\overline{a_1},\ldots,\overline{a_n})$ (as in \cite{israeliani-1}).

\end{remark}

{From Lemma 2.1 and Proposition 2.2 in \cite{israeliani-1}}, taking into account Remark \ref{idealicorrispondenti}, we directly get: 
	
	\begin{proposition}[{\cite[Lemma 2.1 and Proposition 2.2]{israeliani-1}}]\label{idealeproprio}
	Let $(a_1,\ldots,a_n)\in {\mathbb{H}}^n$. 
	\begin{enumerate}
		\item  If $a_la_m=a_ma_l$ for any $1\le l,m \le n$,  then $\mathcal I_{(a_1,\ldots,a_n)}$ is a maximal ideal in $\mathbb{H}[q_1,\ldots,q_n]$;
		\item if there exist $l,m \in\{1,\ldots, n\}$ such that $a_la_m\neq a_ma_l$, 
		then $\mathcal I_{(a_1,\ldots,a_n)}=\mathbb{H}[q_1,\ldots,q_n]$.
	\end{enumerate} 
	\end{proposition}

%
%

Moreover, recalling Formula \eqref{valori} and Remark \ref{idealicorrispondenti}, we immediately obtain:
\begin{proposition}[{\cite[Proposition 2.2]{israeliani-1}}]\label{idealemassimale1}
	Let $(a_1,\ldots,a_n)\in {\mathbb{H}}^n$ with $a_la_m=a_ma_l$ for any $1\le l,m \le n$,  and let $\mathcal I_{(a_1,\ldots,a_n)}$ be
	the right ideal generated by $q_1-a_1, \ldots, q_n-a_n$.
	Then a regular polynomial $P\in \mathbb{H}[q_1,\ldots,q_n]$ belongs to $\mathcal I_{(a_1,\ldots,a_n)}$ if and only if $P(a_1,\ldots,a_n)=0$.
\end{proposition}
}

Let us now investigate what kind of information we obtain on the $*$--factorization of  a regular polynomial which vanishes at a generic point $(a_1,\ldots,a_n)$, without any assumption on the commutativity of its components.

\begin{proposition}\label{zeri}
	A regular polynomial in $P\in
	\mathbb{H}[q_1,\ldots,q_n]$ vanishes at $(a_1,\ldots,a_n)\in \mathbb{H}^n$ if and only if there exist
	$P_k\in \mathbb H[q_1,\ldots,q_k]$ for any $k=1,\ldots,n$ such that
	\begin{equation*}
P(q_1,\ldots,q_n)=\sum_{k=1}^n(q_k-a_k)*P_k(q_1, \ldots,q_k).
	\end{equation*}	
\end{proposition}
\begin{proof} We start performing the $*$-division of $P$ by $(q_n-a_n)$ as in Proposition \ref{euclidean}; we thus obtain 
	\[P(q_1,\ldots,q_n)=(q_n-a_n)*P_n(q_1,\ldots,q_n)+R_n(q_1,\ldots,q_{n-1}).\]
	Notice that $R_n(q_1,\ldots, q_{n-1})$ does not depend on $q_n$ since
	$\deg_{q_n}R<1$. If we now divide $R_{n}$ by $(q_{n-1}-a_{n-1})$, we obtain
	\[P(q_1,\ldots,q_n)=(q_n-a_n)*P_n(q_1,\ldots,q_n)+(q_{n-1}-a_{n-1})*P_{n-1}(q_1,\ldots,q_{n-1})+R_{n-1}(q_1,\ldots,q_{n-2}).\]
	Iterating this process, at the $(n-2)$-th step we get
\begin{align*}
P(q_1,\ldots,q_n)&=(q_n-a_n)*P_n(q_1,\ldots,q_n)+(q_{n-1}-a_{n-1})*P_{n-1}(q_1,\ldots,q_{n-1})+\\
&+\cdots+(q_{2}-a_{2})*P_{2}(q_1,q_2)+R_{2}(q_1)
\end{align*}
where $R_{2}$ is a one-variable regular polynomial in $q_1$.
Thanks to Proposition \ref{mlineare}, evaluating at $(a_1,\ldots,a_n)$ the previous equality gives
\[0=P(a_1,\ldots,a_n)=R_{2}(a_1)\]
which, recalling the one-variable theory (see \cite[Proposition 3.18]{libroGSS}), 
implies that there exists $P_1\in \mathbb{H}[q_1]$ such that $R_{2}(q_1)=(q_1-a_1)*P_1(q_1)$.
Therefore we prove the statement.
\end{proof}
 Let $(a_1,\ldots,a_n)\in \mathbb{H}^n$, and let us denote
  by $E_{(a_1,\ldots,a_n)}$ the set of regular polynomials in
  $\mathbb{H}[q_1,\ldots,q_n]$ which vanish at $(a_1,\ldots,a_n)$,
  namely
	\[E_{(a_1,\ldots,a_n)}:=\left\{\sum_{k=1}^n(q_k-a_k)*P_k(q_1,\ldots,q_k) \ : P_k \in \mathbb{H}[q_1, \ldots,q_k] \text{ for any $k=1,\ldots,n$}\right\}.\] 

        \begin{proposition}\label{idealeono}
		Let $(a_1,\ldots,a_n) \in \mathbb{H}^n$. The set $E_{(a_1,\ldots,a_n)}$ of regular
                polynomials vanishing at $(a_1,\ldots,a_n)$ is an
                ideal if and only if $a_la_m=a_ma_l$ for any
                $l,m=1,\ldots,n$.
	\end{proposition}

        \begin{proof}
		If $(a_1,\ldots,a_n)\in \mathbb{H}^n$ is such that
                $a_la_m=a_ma_l$ for any $l,m=1,\ldots,n$, then
                $E_{(a_1,\ldots,a_n)}$ coincides with the ideal
                $\mathcal I_{(a_1,\ldots,a_n)}$. Indeed
                $E_{(a_1,\ldots,a_n)} \subseteq \mathcal
                I_{(a_1,\ldots,a_n)}$ by definition. On the other
                hand, by Proposition \ref{idealemassimale1} all regular
                polynomials in $\mathcal{I }_{(a_1,\ldots,a_n)}$
                vanish at $(a_1,\ldots,a_n)$ and hence $\mathcal
                I_{(a_1,\ldots,a_n)} \subseteq E_{(a_1,\ldots,a_n)}$.
		
		Let now $(a_1,\ldots,a_n)\in \mathbb{H}^n$ be such
                that $a_la_m\neq a_ma_l$ for some
                $l,m\in\{1,\ldots,n\}$ and suppose that
                $E_{(a_1,\ldots,a_n)}$ is an ideal. Since the regular
                polynomials $q_1 -a_1, \ldots,q_n - a_n$ belong to
                $E_{(a_1,\ldots,a_n)}$, the same holds for the ideal
                generated by them $\mathcal{I
                }_{(a_1,\ldots,a_n)}=\langle q_k-a_k \ :
                \ k=1,\ldots,n \rangle$. Recalling Proposition
                \ref{idealeproprio}, $\mathcal{I
                }_{(a_1,\ldots,a_n)}=\mathbb H[q_1,\ldots,q_n]
                \subseteq E_{(a_1,\ldots,a_n)}$, a contradiction.
		
	\end{proof}
	%
	%
	\begin{remark}\label{zeritrasmessi}
			        The previous Proposition can be read as follows: if a
        regular polynomial $P$ vanishes at $(a_1,\ldots,a_n)$ with
        $a_la_m=a_ma_l$ for any $l,m=1,\ldots,n$, then $P*Q$ still
        vanishes at $(a_1,\ldots,a_n)$ for any $Q$ in $\mathbb
        H[q_1,\ldots,q_n]$. The same is not in general true if a zero
        does not have commuting components.  
        \end{remark}
        The previous remark is also relevant
        when considering the zeros of the symmetrization of a regular 
        polynomial.  Recall that if a slice regular function in one
        quaternionic variable vanishes at $a \in \mathbb{H}$, then
        also its symmetrization does.  In several quaternionic variables, we can
        show that the situation is the same only assuming the
        commutativity of the components of the assigned zero.
	\begin{corollary}\label{zerisym}
		If $P\in \mathbb{H}[q_1, \ldots, q_n]$ vanishes at
                $(a_1,\ldots, a_n)$, with $a_la_m=a_ma_l$ for any
                $l,m=1,\ldots,n$, then also its symmetrization does
                since $P^s=P*P^c$.
	\end{corollary}
	This is not necessarily true if  $a_la_m \neq a_ma_l$ for some $l,m \in \{1,\ldots,n\}$. 
	As an example, consider
	$P(q_1,q_2)=q_1q_2-k$, whose symmetrization is $P^s(q_1,q_2)=q_1^2q_2^2+1$. We have $P(i,j)=0$, while $P^s(i,j)=2.$

\subsection{The two variables case}

Using the properties of regular polynomials in one quaternionic
variable, we are able to prove several results concerning the
vanishing sets of regular polynomials in two  quaternionic variables which enlighten
different interesting phenomena passing from one to several quaternionic
variables.

We begin by assigning the value of the first variable.
The next proposition reminds the geometric origin of spherical zeros for slice regular functions in one quaternionic variable.

  \begin{proposition}
    Let $a\in\mathbb{H}\setminus\mathbb{R}$.
If $P(q_1,q_2)$ vanishes on $\{a\}\times C_a$ and also at $(a,b)$, with $b\notin C_a$, then $P$ vanishes 
on $\{a\}\times \mathbb{S}_b$.
  
  \end{proposition}

\begin{proof}
 Thanks to Proposition \ref{mlineare},

  \[P(q_1,q_2)=(q_1-a)*Q(q_1,q_2)\]
  with $Q(q_1,q_2)=\sum\limits_{ \substack{n=0,\ldots,N  \\ m=0,\ldots,M} } q_1^nq_2^m a_{n m}$.
 Notice that 

  \[(q_1-a)*Q(q_1,q_2)=q_1 Q(q_1,q_2)-aQ(a^{-1}q_1a,a^{-1}q_2a);\]
hence

\begin{eqnarray*}P(a,q_2)  &=&
  a\sum\limits_{ \substack{n=0,\ldots,N  \\ m=0,\ldots,M} } a^nq_2^m a_{n m}-a\sum\limits_{ \substack{n=0,\ldots,N  \\ m=0,\ldots,M} } a^{n-1}q_2^ma a_{n m}\\
  &=& a\sum\limits_{ \substack{n=0,\ldots,N  \\ m=0,\ldots,M} } a^{n-1}[aq_2^m -q_2^ma] a_{n m}.
  \end{eqnarray*}
If $q_2=x+Ky$, put $q_2^m:=\alpha_m+K\beta_m$, with $x,\ y,\ \alpha_m, \beta_m$ real numbers and $K\in\mathbb{S}$.
Similarly, write $a^n:=u_n+Jv_n$ if $a=a_o+Ja_1$ with $a_0,\ a_1,\ u_n,\ v_n$ real numbers and $J\in\mathbb{S}$.
Hence
\begin{eqnarray*}P(a,q_2)  &=&
  a\sum\limits_{ \substack{n=0,\ldots,N  \\ m=0,\ldots,M} } a^{n-1}\underbrace{[aK-Ka]}_{:=T} \beta_ma_{n m}\\
  &=&  a\sum\limits_{ \substack{n=0,\ldots,N  \\ m=0,\ldots,M} } (u_{n-1}+Jv_{n-1}) T \beta_ma_{n m}.
  \end{eqnarray*}
Now
\[T=aK-Ka=(a_o+Ja_1)K-K(a_o+Ja_1)=(JK-KJ)a_1=2J\wedge K a_1;\]
furthermore, since $J\perp T$, then $JT=-TJ$.
Therefore 
\begin{eqnarray}\label{form1}
  P(a,q_2)  &=& a T\sum\limits_{ \substack{n=0,\ldots,N  \\ m=0,\ldots,M} } (u_{n-1}-Jv_{n-1}) \beta_ma_{n m} \nonumber \\
 &=& a T\sum\limits_{ \substack{n=0,\ldots,N  \\ m=0,\ldots,M} } \overline{a^{n-1}} \beta_m a_{n m}.
\end{eqnarray}
 If $P$ vanishes at $(a,b)$ with $b\notin C_a$, then $T\neq 0$ and hence 
$P$ vanishes on $\{a\}\times\mathbb{S}_b$. Indeed (\ref{form1}) is
a function of $\beta_m$ which only depends on $x={\rm Re}(q_2)$ and $y=|{\rm Im}(q_2)|$ and not on the imaginary unit of $q_2$.
\end{proof}

As expected from the Identity Principle, a  regular polynomial in two variables which  vanishes on $\{a\}\times \mathbb{H}$, with $a$ not real, is in fact a regular polynomial in one variable.  
\begin{proposition}
Let   $a\in\mathbb{H}\setminus\mathbb{R}$;
a regular polynomial $P\in\mathbb{H}[q_1,q_2]$
vanishes on $\{a\}\times\mathbb{H}$  if and only if there exists a regular polynomial $ Q\in\mathbb{H}[q_1]$
such that $P(q_1,q_2)=(q_1-a)*Q(q_1)$.

\end{proposition}

{
\begin{proof}
  {Notice that $P$, as a regular polynomial in $\mathbb{H}[q_1,q_2]$, can be written as }
  \begin{equation*}
  P(q_1,q_2)=\sum_{n=0}^{N} q_1^nP_n(q_2)
    \end{equation*}
with $P_n$ regular polynomials in the variable $q_2$.
Since  $P$
vanishes on $\{a\}\times\mathbb{H}$, it turns out that
\begin{equation}\label{f1}
  0\equiv P(a,q_2)=\sum_{n=0}^{N} a^nP_n(q_2)
  \;\ \mathrm{for\ any}\ q_2\in\mathbb{H}.
\end{equation}
If $\deg_{q_1}(P)=1$, then (\ref{f1}) becomes $P_0(q_2)=-aP_1(q_2)\  \mathrm{for\ any}\ q_2\in\mathbb{H}$ , so that
\[P(q_1,q_2)=-aP_1(q_2)+q_1P_1(q_2)\  \mathrm{for\ any}\ q_2\in\mathbb{H}\] 
and since $P(q_1,q_2)-q_1P_1(q_2)\in\mathbb{H}[q_1,q_2]$ is a regular polynomial, it follows that $P_1$ and  hence  $P_0$ are constant.
Furthermore $P(q_1,q_2)=(q_1-a)\cdot C=(q_1-a)*C$ with $C\in\mathbb{H}$.
If $\deg_{q_1}(P)=N>1$, then, as before,
$P_0(q_2)=-aP_1(q_2)-a^2P_2(q_2)-\ldots- a^{N}P_{N}(q_2)$  $\mathrm{for\ any}\ q_2\in\mathbb{H}$, so that
\[P(q_1,q_2)=
-aP_1(q_2)-a^2P_2(q_2)-\ldots- a^{N}P_{N}(q_2)+
q_1P_1(q_2)+q_1^2P_2(q_2)+\ldots +q_1^{N}P_{N}(q_2).\]
Since $P(q_1,q_2)-q_1P_1(q_2)-q_1^2P_2(q_2)-\ldots -q_1^NP_N(q_2)\in\mathbb{H}[q_1,q_2]$ is a regular polynomial, it follows that
\[-a[P_1(q_2)-aP_2(q_2)-\ldots- a^{N-1}P_{N}(q_2)]\]
is constant, or $P_1(q_2)-aP_2(q_2)-\ldots- a^{N-1}P_{N}(q_2)\equiv C_1 \; { \rm for\ any }\ q_2\in\mathbb{H},$ hence
\[C_1-P_1(q_2)\equiv -a[P_2(q_2)+\ldots+ a^{N-2}P_{N}(q_2)] \;\ \mathrm{for\ any  }\ q_2\in\mathbb{H}.\]
Since $C_1-P_1(q_2)$ is a regular polynomial,
 it follows that
\[P_2(q_2)+aP_3(q_2)-\ldots- a^{N-2}P_{N}(q_2)\]
is constant and so $P_1$ is constant as well.
After iterating the same procedure, we eventually get that $P_2$, $P_3, \ldots ,\ P_N$ are constant  regular polynomials. 
Therefore $P(q_1,q_2)\in\mathbb{H}[q_1]$ is a regular polynomial in the variable $q_1$ which vanishes at $q_1=a$, so, recalling \cite[Proposition 3.18]{libroGSS}, that
there exists $Q\in\mathbb{H}[q_1]$ such that
\[P(q_1,q_2)=(q_1-a)* Q(q_1).\]
\end{proof}

Let us now consider zeros of a regular polynomial with assigned second
component

\begin{proposition}\label{timesU}
	Let $P\in \mathbb{H}[q_1,q_2]$. If there exist $J\in \mathbb
        S$, a subset $U\subset \mathbb{C}_J$ with an accumulation
        point in $\mathbb{C}_J$ and $b\in \mathbb{H}$ such that $P$
        vanishes on $U\times\{b\}$, then $P$ vanishes on
        $\mathbb{H}\times\{b\}$.
\end{proposition}
\begin{proof}
  Let $P(q_1,q_2)=\sum\limits_{ \substack{n=0,\ldots,N
      \\ m=0,\ldots,M} }{q_1}^n{q_2}^ma_{n,m}$.  Then
  \[Q(q_1):=P(q_1,b)=\sum_{ \substack{n=0,\ldots,N  \\ m=0,\ldots,M} }{q_1}^n{b}^ma_{n,m}\]
  is a regular polynomial in one quaternionic variable vanishing on a
  subset $U\subset \mathbb{C}_J$ with an accumulation point in
  $\mathbb{C}_J$. Thanks to the Identity Principle for slice regular
  functions in one quaternionic variable \cite[Theorem 1.12]{libroGSS}, $Q$ is identically zero on
  $\mathbb{H}$, and hence $P$ vanishes on $\mathbb{H}\times\{b\}$.
\end{proof}
Similarly, using the Strong Identity Principle for slice regular
functions in one quaternionic  variable \cite[Corollary 3.14]{libroGSS}, we can also prove the following stronger statement.
\begin{proposition}
  Let $P\in \mathbb{H}[q_1,q_2]$.  If there exist a two dimensional
  sphere $S=x+y\mathbb S\subset \mathbb{H}$ and $b\in \mathbb H$ such
  that $P$ vanishes on $U \times \{b\}$ where $U\subset \mathbb{H}$ is
  such that $U\cap S=\varnothing$ but $U$ has an accumulation point on
  $S$,
then $P$ vanishes on $\mathbb{H}\times\{b\}$.
\end{proposition}

\noindent

%

We want to describe ideals contained in the set  $E_{(a,b)}$ of  regular polynomials vanishing at a point $(a,b) \in \mathbb{H}^2$.
First we need the following
\begin{proposition}\label{Sa}
	A regular polynomial in two variables $P\in
	\mathbb{H}[q_1,q_2]$ vanishes on $\mathbb{S}_a\times \{b\}$ if and only
	if there exist $P_1\in \mathbb{H}[q_1]$ and $P_2\in \mathbb{H}[q_1,q_2]$ such that
	\begin{align*}
	P(q_1,q_2)&=(q_1^2-2{\rm Re}(a)q_1+|a|^2)*P_1(q_1)+(q_2-b)*P_2(q_1,q_2) \\	
\end{align*}	
	\end{proposition}
\begin{proof}
Dividing $P$ by $(q_2-b)$ as in the proof of Proposition \ref{zeri},
we can write  \[P(q_1,q_2)=(q_2-b)*P_2(q_1,q_2)+R(q_1),\] 
and hence  
$R(q_1)=P(q_1,q_2)-(q_2-b)*P_2(q_1,q_2)$ vanishes on $\mathbb{S}_a\times \{b\}$. Thus, recalling the one-variable theory (see \cite[Proposition 3.18]{libroGSS}),
$$R(q_1)=(q_1^2-2{\rm Re}(a)q_1+|a|^2)*P_1(q_1).$$ 
\end{proof}
Let us denote by $E_{\mathbb{S}_a\times \{b\}}$ the set of regular polynomials vanishing on $\mathbb{S}_a\times \{b\}$, namely
\[E_{\mathbb{S}_a\times \{b\}}=\{(q_1^2-2{\rm Re}(a)q_1+|a|^2)*P_1(q_1)+(q_2-b)*P_2(q_1,q_2) \ : \ P_1\in \mathbb{H}[q_1], P_2 \in \mathbb{H}[q_1,q_2]\}.\]
\begin{proposition}\label{Said}
The set $E_{\mathbb{S}_a\times \{b\}}$ is an ideal contained in $E_{(a,b)}$.
\end{proposition}	
\begin{proof}
	The fact that $(q_1^2-2{\rm Re}(a)q_1+|a|^2)$ is a regular polynomial only in the first variable, with real coefficients, vanishing identically on $\mathbb{S}_a$, guarantees that 
$$(q_1^2-2{\rm Re}(a)q_1+|a|^2)*Q=(q_1^2-2{\rm Re}(a)q_1+|a|^2)\cdot Q$$ vanishes on $\mathbb{S}_a\times \mathbb{H}$ for any $Q\in \mathbb H[q_1,q_2]$. 
Hence 
any regular polynomial of the form $$(q_1^2-2{\rm Re}(a)q_1+|a|^2)*Q_1+(q_2-b)*Q_2,$$ with $Q_1,Q_2 \in \mathbb H[q_1,q_2]$, vanishes on $\mathbb{S}_a\times \{b\}$. Therefore $E_{\mathbb{S}_a\times \{b\}}$ is an ideal. Since $a\in \mathbb{S}_a$, $E_{\mathbb{S}_a\times \{b\}}\subset E_{(a,b)}$.
\end{proof}
Let us now show the following
\begin{proposition}\label{idealisfere}
	Let $a,b\in \mathbb H$ be such that $ab\neq ba$ and let
 $I$ be an ideal of $\mathbb{H}[q_1,q_2]$ contained in $E_{(a,b)}$. Then  $I\subseteq E_{\mathbb{S}_a\times \{b\}}$.
\end{proposition}
\begin{proof}
	Let $P\in I$. Then $P(a,b)=0$. Consider now $P*q_2$ which still belongs to $I$. Hence
	\[0=P(q_1,q_2)*{q_2}_{|_{(a,b)}}=q_2*P(q_1,q_2)_{|_{(a,b)}}=q_2P(q_2^{-1}q_1q_2,q_2)_{|_{(a,b)}}=bP(b^{-1}ab,b).\]
	Thus the regular  polynomial $P(\cdot,b)\in \mathbb H[q_1]$ vanishes at $q_1=a$ and at $q_1=b^{-1}ab$ which are two different points on the two sphere $\mathbb{S}_a$. By the one-variable theory (see \cite[Theorem 3.1]{libroGSS}), $P(\cdot,b)$ vanishes on the entire sphere $\mathbb{S}_a$. Hence $P \in E_{\mathbb{S}_a\times \{b\}}$.
\end{proof}
}Propositions \ref{Sa} and \ref{Said}, and \ref{idealisfere} will be
used in the next section in order to study the vanishing set of regular
polynomials in $\mathbb{H}[q_1,q_2];$ the first two results can be
generalised easily to regular polynomials in $n$ variables vanishing on sets
of the form $\mathbb S_{a_1}\times \{a_2\}\times \cdots
\times\{a_n\}$, where $a_la_m=a_ma_l$ for all $l,m$ such that $2\leq
l,m\leq n$.



\section{Nullstellensatz type theorems for regular polynomials}\label{WeNu}

In this Section, we further investigate the relations between zero
sets of regular polynomials and ideals in $\mathbb{H}[q_1,\ldots,q_n]$.  In
the complex setting such correspondence is established by the Hilbert
Nullstellensatz. This result admits two equivalent formulations. The
first one, known as ``Weak Nullstellensatz", states that $I$ is a
proper ideal of $\mathbb{C}[z_1,\ldots,z_n]$ if and only if there
exists a common zero of all  regular polynomials in $I$. 
The second one, the ``Strong Nullstellensatz", is more abstract and
involves the notion of {\em radical} of an ideal. The equivalence of
the two statements deeply relies on the fact that point-evaluation is
a homomorphism (see, e.g., \cite{lang}). As already pointed out in
Remark \ref{valutazione}, this is not the case in the quaternionic
setting. Using a proper strategy, Alon and Paran in
\cite{israeliani-1} prove a Weak and a Strong Nullstellensatz for
quaternionic polynomials with central variables. Thanks to the isomorphism $\varphi$  introduced in \eqref{iso}  it is immediate to rephrase in our setting the Weak Nullstellensatz proven by Alon and Paran \cite[Theorem 1.1]{israeliani-1}.

{

\begin{theorem}[Weak Nullstellensatz for regular polynomials]
	\label{WN}
	Let $I$ be a proper right ideal of
	$\mathbb{H}[q_1,\ldots,q_n]$. Then there exists a point $(a_1,\ldots,a_n)\in
	(\mathbb{C}_J)^n$, for some $J \in \mathbb{S}$, such that every regular polynomial
	in $I$ vanishes at $(a_1,\ldots,a_n)$.
	
\end{theorem}


%
%

}
To discuss a Strong version of the Nullstellensatz for regular polynomials
we need to introduce some additional notation.
\begin{definition}
	
	\noindent  Given a right ideal $I$ in $\mathbb{H}[q_1,\ldots,q_n]$,
	we define $\mathcal{V}(I)$ to be the set of common zeros of $P\in I$, i.e.,
	if $Z_P\subset \mathbb{H}^n$ represents the zero set of a regular  polynomial $P\in I$, then 
	\[ \mathcal{V}(I):=\bigcap_{P\in I} Z_P.\]
	
	\noindent Furthermore, we set
	
	\[ \mathcal{V}_c(I):=\mathcal{V}(I)\cap\bigcup_{J\in \mathbb{S}}(\mathbb{C}_J)^n.\]

\end{definition}
Notice that $\mathcal V_c(I)$  is contained in $\mathcal{V}(I)$ and  coincides with the set $Z(I)$ introduced in \cite{israeliani-1}.
Recalling the notation introduced in Section \ref{333}, we
        have that a generic point $(a_1,\ldots,a_n)\in \mathbb{H}^n$ belongs to $\mathcal{ V }(I)$ if
        and only if $I\subseteq E_{(a_1,\ldots,a_n)}$.


Let us now investigate some properties of the two sets $\mathcal{V}(I)$ and $\mathcal{V}_c(I)$ when $I$ is a principal ideal.

\begin{proposition}
Let
	$\langle P\rangle$ be the principal right ideal
	generated by $P\in\mathbb{H}[q_1,\ldots,q_n]$. Then	
	\[\mathcal{ V }_c(\langle P \rangle)=Z_P \cap\bigcup_{J\in \mathbb{S}} (\mathbb{C}_J)^n.\]
	
\end{proposition}
\begin{proof}
	By definition $\mathcal{ V }_c(\langle P \rangle)\subseteq \bigcup_{J\in \mathbb{S}} (\mathbb{C}_J)^n $ and $\mathcal{ V }_c(\langle P \rangle) \subseteq \mathcal{ V }(\langle P \rangle) \subseteq Z_P$. 
	On the other hand, recalling Remark \ref{zeritrasmessi}, $Z_P \cap\bigcup_{J\in \mathbb{S}} (\mathbb{C}_J)^n \subseteq \mathcal{ V }_c(\langle P \rangle)$ and hence they coincide.
	\end{proof}

\noindent Theorem \ref{WN} guarantees that every regular polynomial has at least a zero with commuting components, i.e, $\mathcal{ V }_c(\langle P \rangle)\neq \varnothing$ for any $P\in \mathbb{H}[q_1,\ldots,q_n]$.
%
%
%

\noindent We give an example of a regular polynomial in $\mathbb{H}[q_1,q_2]$ with all the zeros with commuting components, that is such that $\mathcal{ V }_c(\langle P \rangle)=\mathcal{ V }(\langle P \rangle)$.
\begin{example}
	The regular polynomial
	\[P(q_1,q_2)=q_1q_2-1\]
	is such that, if $J\neq K$, then $Z_P\cap (\mathbb{C}_J\times
	\mathbb{C}_K)=\varnothing;$ indeed $Z_P=\{ (q, q^{-1})
	\ :\ q\in\mathbb{H}\setminus \{0\}\}$, which implies that if
	$P(a,b)=0$ there exists $J\in\mathbb{S}$ such that
	$(a,b)\in\mathbb{C}_J\times \mathbb{C}_J.$

\end{example}
\noindent However, not all regular polynomials have all the zeros with commuting components, not even if they have real coefficients 
\begin{example} The regular polynomial
$Q(q_1,q_2)=q_1^2+q_2^2+2$ vanishes at $(i,j)$ (and actually at any pair $(J_1,J_2)$ with $J_1,J_2\in\mathbb{S}$), which implies that $Z_Q\supsetneq \mathcal{V}_c(\langle Q \rangle).$ Moreover $\mathcal{V}(\langle Q \rangle)=Z_Q$, in fact $$Q(q_1,q_2)*q_1^nq_2^m a=q_1^n Q(q_1,q_2)q_2^m a$$ for any monomial $q_1^nq_2^m a$.
\end{example}
\noindent This gives us an example in which $\mathcal{V}(I)\supsetneq \mathcal{V}_c(I).$\\
\\
\noindent  Since the set of regular polynomials vanishing on a given subset $Z$ of $\mathbb{H}^n$ is not in general an ideal (even if $Z$ consists of a single point, as noted in Proposition \ref{idealeono}), it becomes natural to associate two different ideals with $Z$.
\begin{definition}\label{idealinuovi}
  Let $Z$ be a non-empty subset of $\mathbb{H}^n$.

  \noindent We denote by
                $\mathcal{J}(Z)$  the right ideal generated in $\mathbb{H}[q_1,\ldots,q_n]$
                by regular polynomials which vanish on $Z$,
	\[\mathcal{J}(Z):=\left\{\sum_{k=1}^N P_k*Q_k\ :\ P_k,Q_k\in\mathbb{H}[q_1,\ldots,q_n]\ {\rm with}\ {P_k}_{|_Z}\equiv 0\right\}.\] 
	We denote by  $\mathcal{I}(Z)$ the right ideal 
	\[\mathcal{I}({Z}):=\bigcap_{(a_1,\ldots,a_n)\in Z} \mathcal{I}_{(a_1,\ldots,a_n)}.\]

\end{definition}
Recalling Proposition \ref{idealeono}, note that, in general, neither 
$\mathcal{J}(Z)$ nor $\mathcal{I}({Z})$  coincide with the set of regular polynomials vanishing on $Z$. Let us give an example in two variables.

\begin{example}
Consider the case $Z=\{(i,j)\}$, from Propositions \ref{idealeproprio} 
it follows that 
\[\mathcal{I}(Z)=\mathcal
I_{(i,j)}=\mathbb H[q_1,q_2]\neq E_{(i,j)}.\] 
Moreover, since $q_1-i, q_2-j \in \mathcal{J}(Z)$, we get that $\mathcal{J}(Z)=I_{(i,j)}$ as well.

\end{example}
{
Again by Proposition \ref{idealeono},
given a right ideal $I$ in
$\mathbb{H}[q_1,\ldots,q_n]$, both the ideals $\mathcal{J}(\mathcal{V}_c(I))$
and $\mathcal{I}(\mathcal{V}_c(I))$ coincide with the set of regular
polynomials vanishing on $\mathcal{V}_c(I)$. In particular this yields that we always have the equality
\[\mathcal{J}(\mathcal{V}_c(I))=\mathcal{I}(\mathcal{V}_c(I)).\] 
%
Observe that
\begin{equation}\label{centralvsnon}\mathcal{I}(\mathcal{V}_c(I))
=\bigcap_{(a_1,\ldots,a_n)\in \mathcal{V}_c(I)} \mathcal{I}_{(a_1,\ldots,a_n)}=\bigcap_{(a_1,\ldots,a_n)\in \mathcal{V}(I)} \mathcal{I}_{(a_1,\ldots,a_n)}=\mathcal{I}(\mathcal{V}(I)).
\end{equation}
Moreover,
since $\mathcal{V}_c(I)\subseteq \mathcal{V}(I)$, it turns out that

\begin{equation}\label{inclusioni}
\mathcal{J}(\mathcal{V}(I))\subseteq \mathcal{J}(\mathcal{V}_c(I))= \mathcal{I}(\mathcal{V}_c(I))=\mathcal{I}(\mathcal{V}(I)).
\end{equation}


}

{

\noindent 
To state a version of the Strong Nullstellensatz for regular polynomials, first we recall the following  counterpart of prime ideals in the
noncommutative setting (see \cite{reyes}).
\begin{definition}
	A right ideal $I$ in $\mathbb{H}[q_1,\ldots,q_n]$ is {\em completely prime} if for any $P,Q \in \mathbb{H}[q_1,\ldots,q_n]$ such that $P*Q \in I$ and $P*I\subseteq I$ we have that $P\in I$ or $Q\in I$. 
\end{definition}
The notion of radical of an ideal introduced in \cite{israeliani-1} can be defined also in the setting of regular polynomials.
\begin{definition}
	Let $I$ be a right ideal in $\mathbb{H}[q_1,\ldots,q_n]$. The {\em right radical} $\sqrt{I}$ of $I$ is the intersection of all completely prime right ideals that contain $I$.
\end{definition}
From the isomorphism $\varphi$ introduced in \eqref{iso} and equality \eqref{centralvsnon}, Theorem 4.6 in \cite{israeliani-1},  directly leads to the following version of the Strong Nullstellensatz
for regular polynomials 
\begin{theorem} 
Let $I$ be a right ideal in $\mathbb{H}[q_1,\ldots,q_n]$. Then
\[\mathcal{I}(\mathcal{V}(I))=\sqrt{I}.\] 
\end{theorem}
\noindent In the two-variable case, thanks to the better understanding of the
correspondence between zeros of regular polynomials and ideals, we are able to
give a more concrete version of the Strong Nullstellensatz, stated in
terms of the ideal operator $\mathcal J$ instead of $\mathcal I$.
\begin{definition}
	A subset $D\subseteq \mathbb{H}^2$ is called {\em
          $q_1$-symmetric} if for any $(a,b)\in \mathcal D$, the set
        $\mathbb S_a \times \{b\}$ is contained in $D$.
\end{definition}
\begin{proposition}\label{sym3}
Let $I\in \mathbb{H}[q_1,q_2]$ be an ideal such that $\mathcal V(I)$ is $q_1$-symmetric. Then $\mathcal J(\mathcal V_c(I))$ coincides with the ideal of regular polynomials vanishing on  $\mathcal{V}(I)$.
\end{proposition} 
\begin{proof}Let $(a,b)\in \mathcal V(I)\setminus \mathcal V_c(I)$. Since $\mathcal V(I)$ is $q_1$-symmetric, the set $\mathcal V_c(I)$ contains two
  points $(a_1,b)$ and $(\bar{a_1},b)$ in $\mathbb{S}_a \times \{b\}$
  with $a_1b=ba_1$.  Let $P\in \mathcal J(\mathcal V_c(I))$. Then $P$
  vanishes on $\mathcal V_c(I)$ and hence
  $P(a_1,b)=P(\bar{a_1},b)=0$. Since $P(\cdot,b)$ is a regular
  polynomial in $q_1$, thanks to the one-variable theory we get that
  $P(\cdot,b)$ vanishes on the entire $\mathbb S_a$. Therefore
  $P(q_1,q_2)$ vanishes on $\mathbb{S}_a\times \{b\}$, thus on
  $(a,b)$. Since $(a,b)$ is a generic point in $\mathcal{ V }(I)$, we
  have that $P$ vanishes on $\mathcal{ V }(I).$ 
\end{proof}
\begin{corollary}\label{sym1}
Let $I\in \mathbb{H}[q_1,q_2]$ be an ideal such that $\mathcal V(I)$ is $q_1$-symmetric. Then $\mathcal J(\mathcal V(I))=\mathcal J(\mathcal V_c(I))$.
\end{corollary} 
\begin{proposition} \label{sym2}
	Given an ideal $I\subseteq \mathbb{H}[q_1,q_2]$, either $\mathcal V(I)=\mathcal V_c(I)$, or $\mathcal V(I)$ is $q_1$-symmetric. 
\end{proposition}
\begin{proof} Suppose there exists a point $(a,b)\in \mathcal V(I)\setminus\mathcal V_c(I)$. Since $I\subseteq E_{(a,b)}$ and $ab\neq ba$, thanks to Proposition
\ref{idealisfere}, we get $I\subseteq E_{\mathbb{S}_a\times \{b\}}$ and thus $\mathbb{S}_a\times \{b\}\subseteq \mathcal{V}(I)$.  
\end{proof}

\noindent   The case  $\mathcal V_c(I)=\mathcal V(I)$,
for example, occurs when $I=\langle P\rangle $, with $P=q_1q_2-1$. Since
$Z_P\cap (\mathbb{C}_J\times \mathbb{C}_K)=\varnothing$ for imaginary units $J\neq K$, then every regular 
polynomial in $\langle P \rangle$ vanishes on $Z_P$. Hence
$\mathcal{ V }(\langle P \rangle )=Z_P=\mathcal{ V }_c(\langle P
\rangle )$.

\noindent Combining Proposition \ref{sym3}, Corollary \ref{sym1} and  Proposition \ref{sym2}, we obtain 
\begin{theorem}\label{radJ}
	Let $I$ be  an ideal in $\mathbb{H}[q_1,q_2]$. Then $ \mathcal J(\mathcal V(I))= \mathcal J(\mathcal V_c(I))$; thus $ \mathcal J(\mathcal V(I))=\{P\in \mathbb{H}[q_1,q_2] \ :\ P(q_1,q_2)=0,\ (q_1,q_2)\in \mathcal{V}(I)\}$.

\end{theorem}
\noindent Recalling Equation \eqref{inclusioni} we then have the following 
\begin{theorem}[Strong Nullstellensatz in $\mathbb{H}^2$]\label{SN}
Let $I$ be a right ideal in $\mathbb{H}[q_1,q_2]$. Then
\[\mathcal{J}(\mathcal{V}(I))=\sqrt{I}.\] 	
\end{theorem}

This formulation of the Strong Nullstellensatz has a relevant geometric interpretation since, combining Proposition \ref{sym3} with Theorem  \ref{radJ}, we obtain that $\sqrt I$ coincides with the ideal of regular polynomials vanishing on $\mathcal V(I)$.
}

\medskip
\noindent 
There are several examples suggesting that the equality $ \mathcal J(\mathcal V(I))= \mathcal J(\mathcal V_c(I))$, 
holds also in $\mathbb{H}[q_1,\ldots,q_n]$, with $n>2$. 
This is the key ingredient to give a more geometric interpretation of a version of the Strong Nullstellensatz in several quaternionic variables.

\begin{example}
We here collect examples
in which $\mathcal J(\mathcal V(I))= \mathcal J(\mathcal V_c(I)).$
\begin{enumerate}
\item
	Any ideal  $I$  in $\mathbb{H}[q_1,\ldots q_n]$ such that
	$\mathcal{V}_c(I)=\mathcal{V}(I)$.
	
\item  $I=\langle q_1-a_1\rangle$;
\noindent indeed, $P\in\langle q_1-a_1\rangle$ if and only if $P$
vanishes on $\{a_1\}\times( C_{a_1})^{n-1}=\mathcal{V}_c(\langle q_1-a_1\rangle)$.
Hence $\langle q_1-a_1\rangle=\mathcal{J}(\mathcal{V}_c(\langle
q_1-a_1\rangle))\supseteq\mathcal{J}(\mathcal{V}(\langle
q_1-a_1\rangle))\supseteq \langle q_1-a_1\rangle$.

%
%
%
%

\item $I=\mathcal I_{(a_1,\ldots,a_n)}$:

\noindent  if $a_la_m\neq a_ma_l$ for some $l ,m \in \{1,\ldots, n\}$, then $\mathcal I_{(a_1,\ldots,a_n)}=\mathbb{H}[q_1,\ldots,q_n]$ so $\mathcal{V}(\mathcal I_{(a_1,\ldots,a_n)})=\mathcal{V}_c(\mathcal I_{(a_1,\ldots,a_n)})=\varnothing$.

\noindent If $a_la_m=a_ma_l$ for any $l,m=1,\ldots,n$, then $P\in \mathcal I_{(a_1,\ldots,a_n)}$ if and only if $P$ vanishes at $(a_1,\ldots,a_n)$. Therefore

$\mathcal I_{(a_1,\ldots,a_n)}\supseteq \mathcal{J}(\mathcal{V}_c(\mathcal I_{(a_1,\ldots,a_n)}))\supseteq \mathcal{J}(\mathcal{V}(\mathcal I_{(a_1,\ldots,a_n)}))\supseteq
\mathcal I_{(a_1,\ldots,a_n)}.$

%

  \end{enumerate}
\end{example}

\noindent The authors have in mind to investigate the general statement which might be behind these examples in a forthcoming paper.

\section*{Declarations}
\noindent The authors have no competing interests to declare that are relevant to the content of this article.\\

\end{document}